\documentclass[12pt,letterpaper,twoside]{article}

\usepackage{fullpage}
\usepackage{amsmath}
\usepackage{graphicx}
\usepackage{amssymb}
  \usepackage{amsthm}
\usepackage{epstopdf}
\usepackage[all]{xy}
\usepackage[applemac]{inputenc}
\usepackage[linktocpage=true]{hyperref}
\usepackage{makeidx}
\usepackage{amscd}
\usepackage{fancyhdr}
\usepackage{sectsty}

\allsectionsfont{\normalsize\centering\sc}
\chapterfont{\Large\sc}
\chaptertitlefont{\LARGE\sc}
\subsectionfont{\small\sc}

 

\makeatletter
\def\cleardoublepage{\clearpage\if@twoside \ifodd\c@page\else
\hbox{}
\thispagestyle{empty}
\newpage
\if@twocolumn\hbox{}\newpage\fi\fi\fi}
\makeatother

\newcommand{\Rep}{\textnormal{Rep}}
\newcommand{\Hom}{\textnormal{Hom}}
\newcommand{\inte}{\textnormal{int}}
\newcommand{\I}{\mathbb{I}}

\newcommand{\colim}[1]{\underset{#1}{\textnormal{colim}}\;}
\renewcommand{\lim}[1]{\underset{#1}{\textnormal{lim}}\;}

\newcommand{\hocolim}{\textnormal{hocolim}}
\newcommand{\R}{\mathbb{R}}
\newcommand{\Z}{\mathbb{Z}}
\newcommand{\C}{\mathbb{C}}
\newcommand{\RP}{\mathbb{RP}}

\newcommand{\Grp}{\textbf{Grp}}

\newcommand{\Top}{\textbf{Top}}

\newcommand{\Set}{\textbf{Set}}

\theoremstyle{definition}
\newtheorem{defi}{Definition}[section]
\newtheorem{exam}[defi]{Example}
\newtheorem{remark}[defi]{Remark}

\theoremstyle{remark}

\theoremstyle{plain}\newtheorem{lemma}[defi]{ Lemma}
\newtheorem{prop}[defi]{ Proposition}
\newtheorem{theorem}[defi]{ Theorem}
\newtheorem{corl}[defi]{ Corollary}

\title{\normalsize{\bf COSIMPLICIAL GROUPS AND SPACES OF HOMOMORPHISMS}}
\author{\normalsize\sc{Bernardo Villarreal}}
\date{}

\begin{document}

\maketitle

\begin{abstract}
\noindent {\sc Abstract:} Let $G$ be a real linear algebraic group and $L$ a finitely generated cosimplicial group. We prove that the space of homomorphisms $\Hom(L_n,G)$ has a homotopy stable decomposition for each $n\geq 1$. When $G$ is a compact Lie group, we show that the decomposition is $G$-equivariant with respect to the induced action of conjugation by elements of $G$. In particular, under these hypothesis on $G$, we obtain stable decompositions for $\Hom(F_n/\Gamma^q_n,G)$ and $\Rep(F_n/\Gamma^q_n,G)$ respectively, where $F_n/\Gamma^q_n$ are the finitely generated free nilpotent groups of nipotency class $q-1$.

The spaces $\Hom(L_n,G)$ assemble into a simplicial space $\Hom(L,G)$. When $G=U$ we show that its geometric realization $B(L,U)$, has a non-unital $E_\infty$-ring space structure whenever $\Hom(L_0,U(m))$ is path connected for all $m\geq1$.
\end{abstract}

\section{Introduction}

Let $G$ be a topological group and $\Gamma$ a finitely generated group. The set of homomorphisms $\Hom(\Gamma,G)$ can be identified with the ordered tuples $(\rho(a_1),...,\rho(a_r))$ in $G^r$, where $\rho\colon\Gamma \to G$ is a homomorphism and $a_1,...,a_r$ is a generating set for $\Gamma$. Computing the homotopy type of $\Hom(\Gamma,G)$ has proven to be rather complicated. Nevertheless, there has been recognition of the stable homotopy type in several cases. When $G\subset GL_n(\C)$ is a closed subgroup, A. Adem and F. Cohen gave a homotopy stable decomposition for $\Hom(\Z^n,G)$ as wedges of the quotient spaces $\Hom(\Z^k,G)/S_1(\Z^k,G)$ with $1\leq k\leq n$. Here $S_1(\Z^k,G)$ stands for the $k$-tuples with at least one entry equal to the identity matrix $I$ in $G$. For an arbitrary finitely generated abelian group $\pi$, A. Adem and J. M. G\'{o}mez gave a similar stable decomposition for $\Hom(\pi,G)$, but in this case, $G$ is a finite product of the compact Lie groups $SU(r),Sp(k)$ and $U(m)$. For a real linear algebraic group $G$, we show that this homotopy stable decomposition works for the finitely generated free nilpotent groups $F_n/\Gamma^q_n$. Here $F_n$ denotes the free group on $n$ generators and $\Gamma^q_n$ is the $q$-th stage of its descending central series. To do this, we use a simplicial approach, by noticing that the familly of groups $\{F_n/\Gamma^q_n\}_{n\geq 0}$ fits into a cosimplicial group. 

In a more general setting, if $L\colon \Delta\to \Grp$ is a finitely generated cosimplicial group and $G$ is a topological group, we get the simplicial space $\Hom(L,G)\colon\Delta^{\text{op}}\to \Top$, where $\Hom(L,G)_n:=\Hom(L_n,G)$. We give a homotopy stable decomposition for the $n$-simplices of $\Hom(L,G)$ as follows. Let $X$ be a simplicial space. Define $S^t(X_n)$ as the subspace of $X_n$ in which any element is in the image of the composition of at least $t$ degeneracy maps. We say $X$ is simplicially NDR when all pairs $(S^{t-1}(X_n),S^{t}(X_n))$ are neighborhood deformation retracts. It was proven in \cite{Gitler} that when $X$ is simplicially NDR, each $X_n$ is homotopy stable equivalent to wedges of $S^t(X_n)/S^{t+1}(X_n)$ with $0\leq t\leq n$. When $G$ is a real algebraic linear group, the simplicial space $X=\Hom(L,G)$ is simplicially NDR.  We prove this by showing that the subspaces $S^t(X_n)$ are real affine subvarieties of $X_n$ for all $0\leq t\leq n$ and therefore can be simultaneously triangulated. Denote $S_t(L_n,G):=S^t(\Hom(L_n,G))$.

\begin{theorem}\label{IntT1}
Let $G$ be a linear algebraic group, and $L$ a finitely generated cosimplicial group. For each $n$, there are natural homotopy equivalences
\[\Theta(n)\colon\Sigma \Hom(L_n,G)\simeq\bigvee_{0\leq k\leq n}\Sigma (S_k(L_n,G))/S_{k+1}(L_n,G)).\] 
\end{theorem}

The free groups $F_n$, assemble into a cosimplicial group which we denote by $F$. In this case $\Hom(F,G)$ is $NG$, the nerve of $G$ seen as a toplogical category with one object, which is also the underlying simplicial space of a model of the classifying space $BG$. For each $n$, we take quotients $F_n/K_n$ by normal subgroups $K_n$ that are compatible with coface and codegeneracy homomorphisms of $F$ to get finitely generated cosimplicial groups denoted by $F/K$. The induced simplicial spaces are more easily described since there is a simplicial inclusion $\Hom(F/K,G)\subset NG$.  For each $q>0$, the family $\{\Gamma^q_n\}_{n\geq 0}$ is compatible with $F$. Theorem \ref{IntT1} applied to this family was first conjectured by A. Adem, F. Cohen and E. Torres in \cite{Ad4}, for closed subgroups of $GL_n(\C)$. 

\begin{corl}\label{IntC1}
If $G$ is a Zariski closed subgroup of $GL_n(\C)$, then there are homotopy equivalences for the cosimplicial group $F/\Gamma^q$,
\[\Sigma \Hom(F_n/\Gamma^q_n,G)\simeq\bigvee_{1\leq k\leq n}\Sigma\left(\bigvee^{\binom{n}{k}}\Hom(F_k/\Gamma^q_k,G)/S_1(F_k/\Gamma^q_k,G)\right)\]  
for all $n$ and $q$. 
\end{corl}

For any finitely generated cosimplicial group $L$, conjugation under elements of $G$ gives $\Hom(L_n,G)$ a $G$-space structure. Moreover, if $G$ is a real algebraic linear group, then it has a $G$-variety structure. The subspaces $S^t(\Hom(L_n,G))$ are subvarieties that are invariant under the action of $G$ for all $0\leq t\leq n$. Using techniques from \cite{Park}, when $G$ is a compact Lie group, we show that $\Hom(L_n,G)$ has a $G$-CW-complex structure where each $S^t(\Hom(L_n,G))$ is a $G$-subcomplex. This allows us to prove the equivariant version of the previous theorem. Let $\Rep(L_n,G)$ and $\overline{S_t}(L_n,G)$ denote the orbit spaces of $\Hom(L_n,G)$ and $S^t(\Hom(L_n,G))$ respectively.

\begin{theorem}
Let $G$ be a compact Lie group. Then for each $n$, $\Theta(n)$ in Theorem \ref{IntT1} is a $G$-equivariant homotopy equivalence, and in particular we get homotopy equivalences
\[\Sigma \Rep(L_n,G)\simeq \bigvee_{1\leq k\leq n}\Sigma(\overline{S}_k(L_n,G)/\overline{S}_{k+1}(L_n,G)).\]
\end{theorem}

\noindent Applying this to the cosimplicial group $F/\Gamma^q$ as in Corollary \ref{IntC1}, we obtain 
\[\Sigma \Rep(F_n/\Gamma^q_n,G)\simeq \bigvee_{1\leq k\leq n}\Sigma\left(\bigvee^{\binom{n}{k}}\Rep(F_k/\Gamma^q_k,G)/\overline{S_1}(F_k/\Gamma^q_k,G)\right).\]

In the second part of this paper we study the geometric realization of $\Hom(L,G)$ for a finitely generated cosimplicial group, which we denote by $B(L,G)$. We show that the set of 1-cocycles of $L$ denoted by $Z^1(L)$ is in one to one correspondence with cosimplicial morphisms $F\to L$. With this we show that any 1-cocycle of $L$ defines a principal $G$-bundle over $B(L,G)$. 

When $G=U=\colim{m}U(m)$ we show that $B(L,U)$ has an $\I$-rig structure, that is, if $\I$ stand for the category of finite sets and injections, the functor $B(L,U(\_))\colon\I\to \Top$ is symmetric monoidal with respect to both symmetric monoidal structures on $\I$.  Using the machinery developed in \cite{Ad3} we prove:

\begin{theorem}
Let $L$ be finitely generated cosimplicial group and suppose that the space $\Hom(L_0,U(m))$ is path connected for all $m\geq 1$. Then, $B(L,U)$ is a non-unital $E_\infty$-ring space.
\end{theorem}

This theorem is also true if we replace $U$ by $SU,Sp,SO$ or $O$.\\

{\bf Acknowledgements:} I would like to thank A. Adem for his supervision and advice throughout this work. O. Antol\'{i}n Camarena for his useful input within the details of this paper. I would also like to thank M. Bergeron, F. Cohen and J.  M. G\'{o}mez, for their comments on an earlier version.

\section{Homotopy Stable Decompositions}

\subsection[Spaces of Homomorphisms]{Spaces of Homomorphisms}

Let $G$ be a topological group and $\Gamma$ a finitely generated group. Any homomorphism $\rho\colon \Gamma\to G$ is uniquely determined by $(\rho(\gamma_1),...,\rho(\gamma_n))\in G^n$ when $\gamma_1,...,\gamma_n\in \Gamma$ is a set of generators. On the other hand, if we fix a presentation of $\Gamma$, then an $n$-tuple $(g_1,...,g_n)\in G^n$ will induce an element in $\Hom(\Gamma,G)$ whenever $\{g_i\}_{i=1}^n$ satisfy the relations in the presentation of $\Gamma$. Thus, there is a one to one correspondence between the subset of such $n$-tuples in $G^n$ and $\Hom(\Gamma,G)$. Topologize $\Hom(\Gamma,G)$ with the subspace topology on $G^n$.

\begin{lemma}\label{lemma2}
Let $\varphi\colon\Gamma\to\Gamma^\prime$ be a homomorphism of finitely generated groups. If $G$ is a topological group, then $\varphi^*\colon \Hom(\Gamma^\prime,G)\to \Hom(\Gamma,G)$ is continuous.
\end{lemma}
\begin{proof}
Suppose $\Gamma=\langle a_1,...,a_r\;|\;R\rangle$ and $\Gamma^\prime=\langle b_1,...,b_m\;|\;R^\prime\rangle.$ Recall that the induced map $\varphi^*\colon \Hom(\Gamma^\prime,G)\to \Hom(\Gamma,G)$ is given by 
\[(\rho(b_1),...,\rho(b_m))\mapsto(\rho(\varphi(a_1)),...,\rho(\varphi(a_r)))\]
for $\rho\colon\Gamma^\prime\to G$. For any $i$, $\varphi(a_i)=b_{i_1}^{n_{i_1}}\cdots b_{i_{q_i}}^{n_{i_{q_i}}}$. By fixing one presentation for each $\varphi(a_i)$ we get that $\varphi^*$ is given by
\begin{align*}
(\rho(b_1),...,\rho(b_m))\mapsto&(\rho(b_{1_1}^{n_{1_1}}\cdots b_{1_{q_1}}^{n_{1_{q_1}}}),...,\rho(b_{r_1}^{n_{r_1}}\cdots b_{r_{q_r}}^{n_{r_{q_r}}}))=\\
&(\rho(b_{1_1})^{n_{1_1}}\cdots\rho( b_{1_{q_1}})^{n_{1_{q_1}}},...,\rho(b_{r_1})^{n_{r_1}}\cdots \rho(b_{r_{q_r}})^{n_{r_{q_r}}}).
\end{align*}
Therefore $\varphi^*$ is the restriction of the map $G^m\to G^r$ given by
\[(g_1,...,g_m)\mapsto(g_{1_1}^{n_{1_1}}\cdots g_{1_{q_1}}^{n_{1_{q_1}}},...,g_{r_1}^{n_{r_1}}\cdots g_{r_{q_r}}^{n_{r_{q_r}}})\]
which is continuous.
\end{proof}

In particular, this Lemma tells us that given any two presentations of $\Gamma$, we get an isomorphism $\varphi\colon\Gamma\to\Gamma$ and hence a homeomorphism $\varphi^*$ between the induced spaces of homomorphisms. Therefore the topology on the space of homomorphisms does not depend on the choice of presentations.

Recall that an affine variety is the zero locus in $k^n$ of a family of polynomials on $n$ variables over a field $k$. Throughout this paper we will focus only on $k=\R$. An affine variety that has a group structure with group operations given by polynomial maps, i.e., maps $f=(f_1,...,f_n)$ where each $f_i$ is a polynomial, is called a linear algebraic group. For example, consider any matrix group. It is easy to check that matrix multiplication is in fact a polynomial map. For the inverse operation of matrices, it is easier to think of matrix groups as subgroups of $SL(n,\R)$. Any matrix $A$ in $SL(n,\R)$ satisfies $A^{-1}=C^t$, the transpose of the cofactor matrix $C$ of $A$. Since the cofactor matrix is described only in terms of minors of $A$, the map $A\mapsto C^t$ is a polynomial map. In fact, this is the general example, since it can be shown that any linear algebraic group is a group of matrices. 

\begin{lemma}\label{lemma1}
Let $G$ be a linear algebraic group, then for any finitely generated group $\Gamma$, $\Hom(\Gamma,G)$ is an affine variety. Moreover, if $\varphi$ is a homomorphism of finitely generated groups, then $\varphi^*$ is a polynomial map.
\end{lemma}
\begin{proof}
Suppose $\Gamma$ is generated by $\gamma_1,\gamma_2,...,\gamma_r$ and has a presentation $\{p_\alpha\}_{\alpha\in\Lambda}$. Each $p_\alpha$ is of the form $\gamma_{k_1}^{n_1}\cdots\gamma_{k_q}^{n_q}=e$, $n_j\in \Z$ and $\gamma_{k_l}\in \{\gamma_1,...,\gamma_r\}$ for all $1\leq j\leq q$. For any homomorphism $\rho\colon\Gamma\to G$ and any such relation $p_\alpha$ we have  \[\rho(p_\alpha)=\rho(\gamma_{i_1}^{n_1}\cdots\gamma_{i_q}^{n_q})=\rho(\gamma_{i_1})^{n_1}\cdots\rho(\gamma_{i_q})^{n_q}=I,\] the identity matrix in $G$. Since products and inverses in $G$ are given in terms of polynomials, this sets up a family of polynomial relations $\{y_{\alpha,i,j}\}_{\alpha,i,j}$, where each $y_{\alpha,i,j}$ is induced by $\rho(p_\alpha)_{i,j}=\delta_{ij}$, the $i,j$ entry of the matrix equality $\rho(p_\alpha)=I$. These relations do not depend on $\rho$, only on $p_\alpha$, in the sense that any $r$-tuple $(g_1,...,g_r)\in G$ satisfying $\{y_{\alpha,i,j}\}_{\alpha,i,j}$, i.e.
\[(g_{i_1}^{n_1}\cdots g_{i_q}^{n_q})_{i,j}=\delta_{ij}\]
for all $\alpha\in\Lambda$ and $1\leq i,j\leq n$, is an element of $\Hom(\Gamma,G)$. Adding the polynomial relations $\{y_{\alpha,i,j}\}_{\alpha\in\Lambda}$ to the ones describing $G^r$ define $\Hom(\Gamma,G)$ as an affine variety. 

For the second part, recall from the proof of Lemma \ref{lemma2}, that $\varphi^*$ is defined in terms of products and inverses of matrices and thus a polynomial map. 
\end{proof}

Similarly, this Lemma tells us that the affine variety structure on $\Hom(\Gamma,G)$ does not depend on the presentation of $\Gamma$. Indeed, any isomorphism of groups will induce an isomorphism of affine varieties.

\subsection{Triangulation of Semi-algebraic Sets}

\begin{defi}
A real semi-algebraic set is a finite union of subsets of the form
\[\{x\in\R^n\;|\;f_{i}(x)>0,\;g_{j}(x)=0\text{ for all }i,j\},\]
where $f_i(x)$ and $g_j(x)$ are a finite number of polynomials with real coefficients. 
\end{defi}

Using Hilbert's basis theorem, all affine varieties over $\R$ are real semi-algebraic sets. Indeed, the zero locus ideal of an affine variety will be finitely generated and thus the affine variety can be carved out by finitely many polynomials.

What makes semi-algebraic sets more interesting is that images of semi-algebraic sets in $\R^n$ under a polynomial map $\R^n\to\R^m$ are semi-algebraic sets in $\R^m$ (see \cite[p.~167]{hironaka1}), as opposed to affine varieties and regular maps.

 Let $M$, $N$ be semi-algebraic subsets of $\R^m$ and $\R^n$, respectively. A continuous map $f\colon M\to N$ is said to be semi-algebraic if its graph is a semi-algebraic set in $\R^m\times\R^n$. The next result is proven in \cite[~p. 170]{hironaka1}.

\begin{prop}\label{Trian1}
Given a finite system of bounded semi-algebraic sets $M_i$ in $\R^n$, there is a simplicial complex $K$ in $\R^n$ and a semi-algebraic homeomorphism $k\colon |K|\to \bigcup \overline{M_i}$ where each $M_i$ is a finite union of $k(\inte|\sigma|)$'s with $\sigma\in K$.
\end{prop}

\begin{remark}\label{Trian2}
Proposition \ref{Trian1} can be stated without the boundedness condition and the details can be found in \cite[~Theorem 2.12]{Park}, where they add the hipotesis $\bigcup M_i$ closed in $\R^n$.
\end{remark}

In the next sections, we will be using this last result in its full extension, but a first application is that any affine variety $Z$ can be triangulated, that is, there exists a simplicial complex $K$ and a homeomorphism $|K|\cong Z$.  With Lemma \ref{lemma1} and Proposition \ref{Trian1} we prove the following.

\begin{corl}
Let $\Gamma$ be a finitely generated group and $G$ a real linear algebraic group. Then $\Hom(\Gamma,G)$ is a triangulated space.
\end{corl}

\subsection{Simplicial Spaces and Homotopy Stable Decompositions}

Let $\Delta$ be category of finite sets $[n]=\{0,1,...,n\}$ with morphisms order preserving maps $f\colon[n]\to[m]$. It can be shown that all morphisms in this category are generated by composition of maps denoted $\text{\bf d}^i\colon [n-1]\to[n]$ and $\text{\bf s}^i\colon[n+1]\to[n]$ where $0\leq i\leq n$. This maps are determined by the relations 
\begin{align*}
\text{\bf d}^j\text{\bf d}^i&=\text{\bf d}^i\text{\bf d}^{j-1}\text{ if }i<j\\
\text{\bf s}^j\text{\bf s}^i&=\text{\bf s}^{i-1}\text{\bf s}^j\text{ if }i>j
\end{align*}
\[\text{\bf s}^j\text{\bf d}^i=\left\{
\begin{matrix}
\text{\bf d}^i\text{\bf s}^{j-1}&\text{if }i<j\\
Id&\text{if }i=j\text{ or }i=j+1\\
\text{\bf d}^{i-1}\text{\bf s}^j&\text{if }i>j+1.
\end{matrix}\right.\]
which are called cosimplicial identities. For any category $\bf{C}$, let ${\bf{C}}^{\text{op}}$ denote its opposite category. A functor \[X\colon\Delta^{\text{op}}\to\Top\] is called a simplicial space.
Here $\Top$ stands for $k$-spaces, i.e., topological spaces where each compactly closed subset is closed. We denote $X_n:=X([n])$ and the maps $d_i=X(\text{\bf d}^i)$ and $s_i=X(\text{\bf s}^i)$ are called face and degeneracy maps respectively.

Fix $n$. Define $S^0(X_n)=X_n$ and for $0<t\leq n$
\[S^t(X_n)=\bigcup_{J_{n,t}}s_{i_1}\circ\cdots\circ s_{i_t}(X_{n-t}),\]
where $s_{i_{j}}\colon X_{n-j}\to X_{n-j+1}$ is a degeneracy map, $1\leq i_1<\cdots< i_t\leq n$ is a sequence of $t$ numbers between $1$ and $n$, and $J_{n,t}$ stands for all possible sequences. This defines a decreasing filtration of $X_n$,
\[S^n(X_n)\subset S^{n-1}(X_n)\subset\cdots\subset S^{0}(X_n)=X_n.\]
For each $n$ there is a homotopy decomposition of $\Sigma X_n$ in terms of the quotient spaces $S^k(X_n)/S^{k+1}(X_n)$ with $k\leq n$. To do this we need the following. 

Let $A\subset Z$ be spaces. Recall that $(Z,A)$ is an NDR pair if there exist continuous functions
\[h\colon Z\times [0, 1]\to Z,\;\;\;u\colon Z\to[0, 1]\] such that the following conditions are satisfied:

1. $A= u^{-1}(0)$,

2. $h(z,0)=z$ for all $z\in Z$,

3. $h(a,t)=a$ for all $a\in A$ and all $t\in [0,1]$, and
 
4. $h(z, 1) \in A$ for all $z\in u^{-1}([0, 1))$.

\noindent Examples of NDR pairs are pairs consisting of CW-complexes and subcomplexes. Indeed, if $Z$ is a CW-complex and $A\subset Z$ a subcomplex, then the inclusion $A\hookrightarrow Z$ is a cofibration which is equivalent to a retraction $X\times I$ to $A\times I\cup X\times\{0\}$ relative to $A\times \{0\}$.

When $X$ is a simplicial space, we call $X$ \emph{simplicially} NDR if $(S^{t-1}(X_n),S^{t}(X_n))$ is an NDR pair for every $n$ and $t\geq 1$. The following result can be found in \cite[~Theorem 1.6]{Gitler}.

\begin{prop}\label{prop15}
Let $X$ be a simplicial space, and suppose $X$ is simplicially \emph{NDR}. Then for every $n\geq0$ there is a natural homotopy equivalence
\[\Theta(n)\colon\Sigma X_n\simeq\bigvee_{0\leq k\leq n}\Sigma(S^k(X_n)/S^{k+1}(X_n)).\] 
\end{prop}

For each $n$, the map $\Theta(n)$ is natural with respect to morphisms of simplicial spaces, that is, natural transformations $X\to Y$.

\subsection{Cosimplicial Groups, 1-cocycles and $\Hom(L,G)$}

\begin{defi}
Let $\Grp$ denote the category of groups. A functor $L\colon\Delta\to\Grp$ is called a \emph{cosimplicial group}. The homomorphisms $d^i=L(\text{\bf d}^i)$ and $s^i=L(\text{\bf s}^i)$ are called coface and codegeneracy homomorphisms respectively. We say that $L$ is a \emph{finitely generated cosimplicial group} if each $L_n$ is finitely generated.
\end{defi}

There are two canonical finitely generated cosimplicial groups that arise from finitely generated free groups. 

\begin{defi}
Define $F\colon\Delta\to\Grp$ as follows: set $F_0=\{e\}$ and for $n\geq1$ let $F_n=\langle a_1,...,a_n\rangle$, the free group on $n$ generators. The coface homomorphisms $d^i\colon F_{n-1}\to F_n$ are given on the generators by 
\[d^0(a_j)=a_{j+1}\]
\[d^i(a_j)=\left\{
\begin{matrix}
a_j&j<i\\
a_ja_{j+1}&j=i\\
a_{j+1}&j>i
\end{matrix}
\right.\;\;\;\text{for }1\leq i\leq n-1\]
\[d^n(a_j)=a_j;\]
and the codegeneracy homomorphisms $s^i\colon F_{n+1}\to F_n$ by
\[s^i(a_j)=\left\{
\begin{matrix}
a_j&j\leq i\\
e&j=i+1\\
a_{j-1}&j>i+1
\end{matrix}
\right.\]
for $0\leq i\leq n$. 
\end{defi}

\begin{defi}
Define $\overline{F}\colon\Delta\to \Grp$ as $\overline{F}_n:=\langle a_0,...,a_n\rangle$ for any $n\geq0$; coface and codegeneracy homomorphisms $\overline{d}^i\colon \overline{F}_{n-1}\to \overline{F}_n$ and $\overline{s}^i\colon \overline{F}_{n+1}\to \overline{F}_n$ respectively, are given on the generators by 
\[\overline{d}^i(a_j)=\left\{\begin{matrix}
a_j&j<i\\
a_{j+1}&j\geq i
\end{matrix}
\right.\;\;\;\;\text{and }\;\;
\overline{s}^i(a_j)=\left\{\begin{matrix}
a_j&j\leq i\\
a_{j-1}&j>i
\end{matrix}
\right.\]
for all $0\leq i\leq n$. 
\end{defi}

\begin{defi}
We will say that a family of normal subgroups $K_n\subset F_n$ is \emph{compatible} with $F$, if $d^i(K_{n-1})\subset K_n$ and $s^i(K_{n+1})\subset K_n$ for all $n$ and all $i$. Similarly we define \emph{compatible} families of $\overline{F}$.
\end{defi}

Given $\{K_n\}_{n\geq0}$ a compatible family with $F$, we get induced homomorphisms
\[\xymatrix{F_{n-1}\ar[r]^{d^i}\ar[d]&F_n\ar[d]\\F_{n-1}/K_{n-1}\ar[r]^{{d^i}}&F_n/K_n}\;\;\;\;\;\;\;\xymatrix{F_{n+1}\ar[r]^{s^i}\ar[d]&F_n\ar[d]\\F_{n+1}/K_{n+1}\ar[r]^{{s^i}}&F_n/K_n.}\]
Define \[F/K \colon\Delta\to\Grp\] as $(F/K)_n=F_n/K_n$ with coface and codegeneracy maps the quotient homomorphisms $d^i$ and $s^i$ respectively. This way $F/K$ is a finitely generated cosimplicial group. Similarly, with a compatible family $\{\overline{K}_n\}_{n\geq0}$ of $\overline{F}$, we can define $\overline{F}/\overline{K}\colon\Delta\to\Grp$.

\begin{exam}\label{example 1}
We describe two families of finitely generated cosimplicial groups that can be constructed using $F/K$ and $\overline{F}/\overline{K}$ through the commutator subgroup. 

$\bullet$ Let $A$ be a group, define inductively $\Gamma^1(A)=A$ and $\Gamma^{q+1}(A)=[\Gamma^q(A),A]$ for $q>1$. The descending central series of $A$ is
\[\Gamma^{q}(A)\trianglelefteq\cdots\trianglelefteq \Gamma^{2}(A)\trianglelefteq \Gamma^{1}(A)=A.\]
Given a homomorphism of groups $\phi\colon A\to B$, $\phi[a,a^\prime]=[\phi(a),\phi(a^\prime)]$ for all $a,a^\prime$ in $A$, so that \[\phi(\Gamma^q(A))\subset \Gamma^q(B).\] Taking $A=F_n$, and denoting $\Gamma^q_n:=\Gamma^q(F_n)$, we have that the family of normal subgroups $\{\Gamma^q_n\}_{n\geq0}$ is compatible with $d_i$ and $s_i$. Thus we can define $F/\Gamma^q$ as $(F/\Gamma^q)_n=F_n/ \Gamma^q_n$ for all $q$ and $1\leq i\leq n$. In particular, for $q=2$, we obtain $F_n/\Gamma^2_n=\Z^n$ for all $n\geq0$.

$\bullet$ Another example using the commutator is the derived series of a group $A$:
\[A^{(q)}\trianglelefteq\cdots\trianglelefteq A^{(1)}\trianglelefteq A^{(0)}=A\]
where $A^{(i+1)}=[A^{(i)},A^{(i)}]$. Again, $\phi(A^{(q)})\subset B^{(q)}$ for any homomorphism $\phi\colon A\to B$. Thus $F/F^{(q)}$, where $(F/F^{(q)})_n=F_n/F_n^{(q)}$ defines a finitely generated cosimplicial group. 

Similarly, $F^{(q)}_{n+1},\Gamma^q_{n+1}\subset \overline{F}_n$ define compatible families of $\overline{F}$ and we obtain the finitely generated cosimplicial groups $\overline{F}/\Gamma^q_{*+1}$ and $\overline{F}/F^{(q)}_{*+1}$.
\end{exam}

\begin{exam}\label{symm}
Here is one example of a cosimplicial group that does not come from a compatible family. Let $L_0=\Sigma_2=\langle \tau\rangle$, $L_1=\Sigma_3=\langle\sigma_1,\sigma_2\rangle$ and define coface homomorphisms
\[
\xymatrix{
L_0 \ar@<-1.05ex>[r]^{d^0} \ar@<1.05ex>[r] ^{d^1}&L_1
}
\]
as $d^0(\tau)=\sigma_2$ and $d^1(\tau)=\sigma_1$. The codegeneracy homomorphism $s_0\colon L_1\to L_0$ is given by $s^0(\sigma_1)=s^0(\sigma_2)=\tau$. This defines a 1-truncated cosimplicial group which we denote by $\Sigma_{2,3}$, that is, a functor $\Sigma_{2,3}\colon\Delta_{\leq 1}\to \Grp$. Here $\Delta_{\leq 1}$ stands for the full subcategory of $\Delta$ with objects $[0]$ and $[1]$. We can extend $\Sigma_{2,3}$ to $\Delta$ by using its left Kan extension. 

For our purposes we describe the second stage of this extension: We have that
\[L_2=\langle a,b,c\:|\;a^2=b^2=c^2,\;aba=bab,\;aca=cac,\;bcb=cbc\rangle,\]
coface homomorphisms
\[
\xymatrix{
L_1 \ar@<-2.1ex>[r]^{d^{0}} \ar[r]^{d^1} \ar@<2.1ex>[r]^{d^2}& L_2
}
\]
are given by 
\begin{align*}
d^0(\sigma_1)&=a&d^1(\sigma_1)&=c&d^2(\sigma_1)&=c\\
d^0(\sigma_2)&=b&d^1(\sigma_2)&=b& d^2(\sigma_2)&=a
\end{align*} 
and codegeneracy homomorphisms
\[\xymatrix{
L_1& L_2  \ar@<-1.05ex>[l]_{s^1} \ar@<1.05ex>[l]_{s^0} 
}\]
by 
\begin{align*}
s^0(a)=s^0(c)&=\sigma_1&s^1(a)=s^1(b)&=\sigma_2\\
s^0(b)&=\sigma_2&s^1(c)&=\sigma_1.
 \end{align*}

\end{exam}

\begin{remark}
The symmetric groups $\Sigma_n$ can not be assembled all together as a cosimplicial group. This is because there are no surjective homomorphisms $\Sigma_n\to\Sigma_{n-1}$ for $n\geq 5$ to use as codegeneracy homomorphisms. Indeed, given a homomorphism $\varphi\colon\Sigma_n\to \Sigma_{n-1}$, $\ker \varphi$ is a normal subgroup of $\Sigma_n$, that is $A_n$ or $\Sigma_n$. Thus the image of $\varphi$ is either the identity element or a subgroup of order 2.
\end{remark}

We describe another method of constructing new cosimplicial groups that arise from a given one. To do this, we recall a concept that was originaly introduced in \cite[~p. 284]{Bous} to define cohomotopy groups (and pointed sets) for a cosimplicial group.

\begin{defi}
Let $L$ be a cosimplicial group. The elemets $b$ in $L_1$ satisfying 
\begin{align}\label{eq1} 
d^2(b)d^0(b)&=d^1(b)
\end{align}
are called 1-cocycles of $L$. The set of 1-cocycles is denoted $Z^1(L)$.
\end{defi}

If $b$ is a 1-cocycle, then applying $s^0$ to equation \ref{eq1}, we obtain $s^0d^2(b)=e$ and using the cosimplicial identities, $d^1s^0(b)=e$, which implies $b\in\ker s^0$. Define inductively $b_n\in L_n$ as $b_{n+1}=d^{n+1}(b_n)$, where $b_1:=b$. These elements will satisfy 
\begin{align}
d^2(b_n)d^0(b_n)=&d^1(b_n) \text{ and}\label{eq2}\\
b_n\in& \ker s^0 \label{eq3}
\end{align}
for all $n\geq 1$. Given a 1-cocycle $b$ we build a new cosimplicial group. \\

{\bf Construction of $L^b$:} Define $L^b\colon\Delta \to \Grp$ as follows. For each $n\geq0 $,  $L^b_n:=\overline{F}_0*L_n$ with codegenarcy homomorphisms $s^i_b:=Id*s^i$, $i\geq 0$. The coface homomorphisms are $d^i_b:=Id*d^i$ for $i>0$. To define $d^0_b$ consider the homomorphism $k_n\colon \overline{F}_0\to \overline{F}_0*L_n$ given by $k_n(a_0)=a_0b_n$ for all $n\geq 0$, then $d^0_b:=k_n*d^0$. There is a canonical inclusion \[\iota_b\colon L\hookrightarrow L^b\] induced by the inclusions $L_n\hookrightarrow \overline{F}_0*L_n$.

\begin{exam}
$\bullet$ When $b=e$, $L^e=\overline{F}_0*L$, where $\overline{F}_0$ represents the constant cosimplicial group with value $\overline{F}_0$.

$\bullet$ Consider the finitely generated free cosimplicial group $F$. The codegeneracy homomorphism $s^0\colon F_1\to F_0$ is the constant map and thus $\ker s^0=F_1=\langle a_1\rangle$. Also \[d^1(a_1)=a_1a_2=d^2(a_1)d^0(a_1),\] and hence $a_1\in Z^1(F)$. Note that any other power of $a_1$ will fail to satisfy the cocycle condition (\ref{eq1}), that is $Z^1(F)=\{e,a_1\}$. Let $F^+=F^{a_1}$. We denote the canonical inclusion as $\iota_+\colon F\hookrightarrow F^+$. A similar argument shows that $Z^1(F/\Gamma^q)=\{e,a_1\}$ for $q>2$. We also denote $(F/\Gamma^q)^{a_1}=F/{\Gamma^q}^+$ and $\iota_+\colon F/\Gamma^q\hookrightarrow F/{\Gamma^q}^+$.

$\bullet$ Consider $F/\Gamma^2$. As in the previous example, $\ker s_0=\langle a_1\rangle$, but since $F_2/\Gamma^2_2=\Z^2$ all powers of $a_1$ will satisfy the cocycle condition, that is, $Z^1(F/\Gamma^2)=\Z$. Thus for each positive $m\in \Z$ we get non isomorphic cosimplical groups $(F/\Gamma^2)^m$ and inclusions $\iota_m\colon F/\Gamma^2\hookrightarrow (F/\Gamma^2)^m$. When $m=1$, we denote $(F/\Gamma^2)^1=F/{\Gamma^2}^+$.

$\bullet$ Consider $\Sigma_{2,3}$ defined in Example \ref{symm}. The product $\sigma_1\sigma_2\in (\Sigma_{2,3})_1$ satisfies
\[d^2(\sigma_1\sigma_2)d^0(\sigma_1\sigma_2)=caab=cb=d^1(\sigma_1\sigma_2)\]
and thus $\sigma_1\sigma_2$ is a 1-cocycle and we get the cosimplicial group $\Sigma_{2,3}^{\sigma_1\sigma_2}$.
\end{exam}

Now we turn our attention to spaces of homomorphisms. For any topological group $G$, its underlying group structure defines the functor \[\Hom_{\Grp}(\_,G)\colon \Grp^{\text{op}}\to \Set.\] If $L$ is a cosimplicial group, the composition of functors $\Hom_{\Grp}(\_,G)L$ which we denote by $\Hom(L,G)$ defines a simplicial set. Whenever $L$ is finitely generated, for each $n$ we can topologize $\Hom(L_n,G)$ in a way that the induced face and degeneracy maps are continuous. Therefore we get the simplicial space $\Hom(L,G)\colon \Delta^{\text{op}}\to \Top$. We list some known simplicial spaces.

 $\bullet$ $\Hom(F,G)=NG$ the nerve of $G$ as a category with one object; 
 
 $\bullet$ $\Hom(F^+,G)=(EG)_*$, Steenrod's model for the total space of the universal principal $G$-bundle $p\colon EG\to BG$, where $p$ is induced by the simplicial map $\iota_+^*\colon \Hom(F^+,G)\to\Hom(F,G)$;
 
 $\bullet$ $\Hom(F/\Gamma^q,G)=(B(q,G))_*$ the underlying simplicial space of the classifying space $B(q,G)$ defined in \cite{Ad5}, \cite{Ad1} (for $q=2$) and \cite{Ad3};
 
 $\bullet$ $\Hom(F/{\Gamma^q}^+,G)=(E(q,G))_*$ the underlying simplicial space of the total space of the universal bundle $p\colon E(q,G)\to B(q,G)$ also defined in \cite{Ad5}, \cite{Ad1} (for $q=2$) and \cite{Ad3}. Again, $p$ is induced by $\iota_+^*\colon\Hom(F/{\Gamma^q}^+,G)\to\Hom(F/\Gamma^q,G)$;
 
 $\bullet$ $\Hom(\overline{F},G)=N\overline{G}$, the nerve of the category $\overline{G}$ that has $G$ as space of objects and there is a unique morphism between any two objects.

\begin{remark} Consider the morphism of cosimplicial groups $\gamma\colon F\to\overline{F}$ given on generators by $\gamma^n(a_i)=a_{i-1}a_i^{-1}$, where $\gamma^n\colon F_n\to \overline{F}_n$. Let $G$ be a topological group. The induced map $\gamma\colon N\overline{G}\to NG$ is the underlying simplicial map of Segal's fat geometric realization model for the universal $G$-bundle.
\end{remark}

\subsection{Homotopy Stable Decomposition of $\Hom(L_n,G)$}

\begin{lemma}\label{Lemma0}
Let $G$ be a linear algebraic group and $L$ a finitely generated cosimplicial group. Let $s^i\colon L_{n+1}\to L_n$ be a codegeneracy map. Then, the image of  $s_i:=(s^i)^*\colon \Hom(L_n,G)\to\Hom(L_{n+1},G)$ is a subvariety for all $0\leq i\leq n$. 
\end{lemma}
\begin{proof}
Suppose $L_{n+1}=\langle a_1,...,a_r\rangle$. The homomorphism $s^i\colon L_{n+1}\to L_n$ is surjective so that $L_n\cong L_{n+1}/\ker s^i$. We can describe $\ker s^i=\langle \{b_\alpha\}_{\alpha\in\Lambda}\rangle$ where each $b_\alpha$ is a fixed product of powers of generators $a_k$. Let $ \rho\colon L_{n+ 1}\to G$ be a homomorphism. Then $(\rho(a_1),...,\rho(a_r))$ is in $ s_i(\Hom(L_n,G))$ if and only if $\rho(b_\alpha)=I$ for all $\alpha\in\Lambda$. That is, these $r$-tuples in $\Hom(L_{n+1},G)$ are determined by the polynomial equations $\{\rho(b_\alpha)=I\}_{\alpha\in\Lambda}$ and hence they build up an affine variety.  
\end{proof}

For a cosimplicial group $L$, denote $S_t(L_k,G):=S^t(\Hom(L_k,G))$.

\begin{theorem}\label{prop16}
Let $G$ be a real algebraic linear group, and $L$ a finitely generated cosimplicial group. Then for each $n$ we have homotopy equivalences
\[\Theta(n)\colon\Sigma \Hom(L_n,G)\simeq\bigvee_{0\leq k\leq n}\Sigma (S_k(L_n,G)/S_{k+1}(L_n,G)).\] 
\end{theorem}
\begin{proof}
Fix $n$. Using Proposition \ref{prop15}, we only need to show that $(S_{t-1}(L_n,G),S_t(L_n,G))$ is a strong NDR-pair for all $0\leq t\leq n$. By Lemma \ref{Lemma0}, each $s_j(\Hom(L_k,G))$ is an affine variety for all $0\leq j,k\leq n$. Then, for all $t\geq 1$ the finite union \[S_t(L_n,G)=\bigcup_{J_{n,t}}s_{i_1}\circ\cdots\circ s_{i_t}(\Hom(L_{n-t},G))\] 
is also an affine variety. Consider the natural filtration
\[S_n(L_n,G)\subset S_{n-1}(L_n,G)\subset\cdots\subset S_{0}(L_n,G)=\Hom(L_n,G).\]
The union $\bigcup_tS_t(L_n,G)=\Hom(L_n,G)$ is an affine variety, and therefore is a closed subspace of some euclidean space. By Remark \ref{Trian2}, $\Hom(L_n,G)$ can be triangulated in a way that each $S_t(L_n,G)$ is a finite union of interiors of simplices. Since $S_t(L_n,G)$ are closed subspaces, it follows that under the triangulation they are subcomplexes. This way the inclusions $S_{t}(L_n,G)\subset S_{t-1}(L_n,G)$ are cofibrations and hence NDR-pairs. Therefore $\Hom(L,G)$ is simplicially NDR.
\end{proof}

\begin{lemma}\label{Lemma1}
Let $G$ be a topological group and consider the cosimplicial group $F/\Gamma^q$. Then
\[S_k(F_n/\Gamma^q_n,G)/S_{k+1}(F_n/\Gamma^q_n,G)\cong \bigvee^{\binom{n}{k}}\Hom(F_k/\Gamma^q_k,G)/S_1(F_k/\Gamma^q_k,G)\]
for all $1\leq k\leq n$.
\end{lemma}

\begin{proof}
Let $1\leq i_1<\cdots< i_{n-k}\leq n$. Consider the projections
\[P_{i_1,...,i_m}\colon G^n\to G^{n-k} \]
given by $(x_1,...,x_n)\mapsto(x_{i_1},...,x_{i_{n-k}})$. We claim that the image of $\Hom(F_n/\Gamma^q_n,G)$ under this projections lies on $\Hom(F_{n-k}/\Gamma^q_{n-k},G)$. Indeed, each projection $P_{i_1,...,i_m}$ is induced by the homomorphism $\varphi\colon F_{n-k}\to F_n$ given on generators as $a_j=a_{i_j}$. Since $\varphi(\Gamma^q_{n-k})\subset \Gamma^q_n$ we get the homomorphism $\overline{\varphi}\colon F_{n-k}/\Gamma^q_{n-k}\to F_n/\Gamma^q_n$ which proves our claim. Assemble the restrictions of $P_{i_1,...,i_m}$ to $\Hom(F_n/\Gamma^q_n,G)$ so that we build up a continuous map
\[\eta_n\colon \Hom(F_n/\Gamma^q_n,G)\to\prod_{J_{n,k}}\Hom(F_{n-k}/\Gamma^q_{n-k},G)\]
given by
\[(x_1,...,x_n)\mapsto \{P_{i_1,...,i_m}(x_1,...,x_n)\}_{(i_1,...,i_{n-k})\in J_{n,k}}\]
where $J_{n,k}$ runs over all possible sequences of length $n-k$, $1\leq i_1<\cdots< i_{n-k}\leq n$. Since all sequences $(i_1,...,i_{n-k})\in J_{n,k}$ are disjoint, the restriction
\[\eta_n|\colon S_k(F_n/\Gamma^q_n,G)\to\bigvee_{J_{n,k}}\Hom(F_{n-k}/\Gamma^q_{n-k},G)/S_1(F_k/\Gamma^q_k,G)\]
has a continuous inverse $\bigvee_{J_{n,k}}s_{j_1}\circ\cdots\circ s_{j_k}$ where $1\leq j_1<\cdots<j_k\leq n$ and the intersection $\{j_1,...,j_k\}\cap\{i_1,...,i_{n-k}\}=\emptyset$. Therefore $\eta_n|$ is a homeomorphism. Finally note that $S_{k+1}(F_n/K_n,G)$ is mapped to $\bigvee S_1(F_{n-k}/\Gamma^q_{n-k},G)$. Taking quotients we get the desired homeomorphism.
\end{proof}

The next corollary was first conjectured in \cite[p.~12]{Ad5} for closed subgroups of $GL_n(\C)$. Since any real linear algebraic group is Zariski closed we have the following version of the conjecture which follows from Theorem \ref{prop16} and Lemma \ref{Lemma1}.

\begin{corl}\label{Corl1}
If $G$ is a Zariski closed subgroup of $GL_n(\C)$, then there are homotopy equivalences for the cosimplicial group $F/\Gamma^q$,
\[\Theta(n)\colon\Sigma \Hom(F_n/\Gamma^q_n,G)\simeq\bigvee_{1\leq k\leq n}\Sigma\left(\bigvee^{\binom{n}{k}}\Hom(F_k/\Gamma^q_k,G)/S_1(F_k/\Gamma^q_k,G)\right)\]  
for all $n$ and $q$. 
\end{corl}

\begin{exam}\label{T1}
Let $G=SU(2)$ and consider $F/\Gamma^q$. 

$\bullet$ The case $q=2$ $(F_n/\Gamma^2_n=\Z^n)$ has been largely studied, and we follow \cite[pp.~482-484]{Ad4}. First a few preliminaries. Let $T\cong S^1$ be a maximal torus of $G$ and $W=N(T)/T=\{[w],e\}$ its Weyl group, where $w=\left(\begin{matrix}0&-1\\
1&0\end{matrix}\right).$ $W$ acts on $T$ via $[w]\cdot t=\overline{wtw^{-1}}=t^{-1}$ and using left translation on $G/T$ we get a diagonal action on $G/T\times T^n$. Let $\mathfrak{t}\cong i\R$ be the Lie algebra of $T$ with the induced action of $W$. There is an equivariant homeomorphism $\mathfrak{t}\to T-\{I\}$ so that $G/T\times_W(T-\{I\})^n\cong G/T\times_W\mathfrak{t}^n.$ The quotient map $G/T\to (G/T)/W\cong\RP^2$ is a principal $W$-bundle and we can take the associated vector bundle $p_n\colon G/T\times_W\mathfrak{t}^n\to\RP^2.$ Let $\lambda_2$ be the canonical vector bundle over $\RP^2$, then we can identify $p_n$ with $n\lambda_2$, the Whitney sum of $n$ copies of $\lambda_2$. The pieces in the homotopy stable decomposition before suspending are
\[\Hom(\Z^n, SU(2))/S_1(\Z^n,SU(2))\cong\left\{\begin{matrix}
S^3&\text{if }n=1\\
(\RP^2)^{n\lambda_2}/s_n(\RP^2)&\text{if }n\geq2\end{matrix}\right.\]
where $(\RP)^{n\lambda_2}$ is the associated Thom space of $n\lambda_2$ and $s_n$ is its zero section. Therefore
\[\Sigma\Hom(\Z^n,SU(2))\simeq\Sigma\bigvee_{n} S^3\bigvee_{2\leq k\leq n}\Sigma\left(\bigvee^{\binom{n}{k}}(\RP^2)^{k\lambda_2}/s_k(\RP^2)\right).\]

$\bullet$ Let $q=3$. An $n$-tuple $(g_1,...,g_n)$ lies in $\Hom(F_n/\Gamma^3_n,SU(2))$ if and only if $[[g_i,g_j],g_k]=I$ for all $1\leq i,j,k\leq n$, i.e., the commutators $[g_i,g_j]$ are central in the subgroup generated by $g_1,...,g_n$. We claim that the center of every non-abelian subgroup of $SU(2)$ sits inside $\{\pm I\}$. Suppose $a,b$ are two elements in $ SU(2)$ such that $[a,b]\ne I$. Then, the ciclic groups $\langle a\rangle$ and $\langle b\rangle$ are contained in different tori $T_1$ and $T_2$ respectively. Since the center of $\langle a,b\rangle$ is abelian it must lie in the intersection $T_1\cap T_2$. These two circles can only intersect at $\{\pm I\}$, which proves our claim. Therefore the central elements $[g_i,g_j]$ are in $ \{\pm I\}$ for all $1\leq i,j\leq n$. Consider
\[B_n(SU(2),\{\pm I\})=\{(g_1,...,g_n)\in SU(2)^n\;|\;[g_i,g_j]\in\{\pm I\}\}\]
the space of almost commuting tuples in $SU(2)$. By the previous observation \[\Hom(F_n/\Gamma^3_n,SU(2))=B_n(SU(2),\{\pm I\}).\]
In \cite[pp.~485-486]{Ad4}, they show that
\[B_n(SU(2),\{\pm I\})/S_1(SU(2),\{\pm I\})\cong \Hom(\Z^n,SU(2))/S_1(\Z^n,SU(2))\bigvee_{K(n)}PU(2)_+\]
where $K(1)=0$ and for $n\geq2$, $K(n)=\frac{7^n}{24}-\frac{3^n}{8}+\frac{1}{12}.$ Here $S_1(SU(2),\{\pm I\})$ are the $n$-tuples in $B_n(SU(2),\{\pm I\})$ with at least one coordinate equal to $I$. Since $PU(2)\cong\RP^3$ and $S_1(SU(2),\{\pm I\})=S_1(F_n/\Gamma^3_n,SU(2))$ we conclude that
\[\Sigma\Hom(F_n/\Gamma^3_n,SU(2))\simeq\Sigma\bigvee_{n} S^3\bigvee_{2\leq k\leq n}\Sigma\left(\bigvee^{\binom{n}{k}}(\RP^2)^{k\lambda_2}/s_k(\RP^2)\bigvee_{K(k)}\RP^3_+\right).\]
\end{exam}

\begin{remark}
For $q\geq 4$ we can find nilpotent subgroups of $SU(2)$ of class $q$. Indeed, if $\xi_n$ is a representative in $SU(2)$ of a primitive $n$-th root of unity, then the subgroup generated by the set $\{\xi_{2q},w\}$ with $w$ as above, is of nilpotency class $q$. With this we can show that the spaces $\Hom(F_n/\Gamma^q_n,SU(2))$ for $q\geq 4$ have more connected components than $\Hom(F_n/\Gamma^3_n, SU(2))$. More details will appear in \cite{Antolin}.
\end{remark}

\subsection{Equivariant Homotopy Stable Decomposition of $\Hom(L_n,G)$}

Let $G,H$ be topological groups and $f\colon G\to H$ a continuous homomorphism. If $L$ is a finitely generated cosimplicial group then for each $n$, we have the commutative diagrams
\[\xymatrix{\Hom(L_n,G)\ar[r]^{d_i}\ar[d]_{f_*}&\Hom(L_{n-1},G)\ar[d]^{f_*}\\
\Hom(L_n,H)\ar[r]^{d_i}&\Hom(L_{n-1},H),}\;\;\;\;\;\xymatrix{\Hom(L_n,G)\ar[r]^{s_i}\ar[d]_{f_*}&\Hom(L_{n+1},G)\ar[d]^{f_*}\\
\Hom(L_n,H)\ar[r]^{s_i}&\Hom(L_{n+1},H)}\]
for all $0\leq i,n$, so that $f_*$ is a simplicial map. Conjugation by elements of $G$ defines a homomorphism $G\to G$ so that $\Hom(L_n,G)$ is a $G$-space and each $S_t(L_n,G)$ is a $G$-subspace. 

\begin{defi}
Let $M$ be a $G$-space. We say that $M$ has a $G$-CW-structure if there exists a pair $(X,\xi)$ such that $X$ is a $G$-CW-complex and $\xi\colon X\to M$ is a $G$-equivariant homeomorphism. 
\end{defi}

We want to show that for all $n$, $\Hom(L_n,G)$ has a $G$-CW complex structure for which $S_n(L_n,G)\subset S_{n-1}(L_n,G)\subset\cdots\subset \Hom(L_n,G)$ are $G$-subcomplexes. To show this we slightly generalize some results in \cite{Park}.

We continue using the techniques of the previous section, so we require $G$ to be a real linear algebraic group. It is known that any compact Lie group has a unique algebraic group structure (see \cite[p.~247]{Onischnick}). 
Assuming $G$ is a compact Lie group, every representation space of $G$ has finite orbit types  (see \cite{Palais}), so when $M$ is an algebraic $G$-variety, the equivariant algebraic embedding theorem \cite[~Proposition 3.2]{Park} implies that $M$ has finite orbit types. Also, this theorem guarantees the existence of a $G$-invariant algebraic map $f\colon M\to \R^d$ for some $d$ such that the induced map $\overline{f}\colon M/G\to f(M)$ is a homeomorphism and $f(M)$ is a closed semi-algebraic set in $\R^d$ (\cite[~Lemma 3.4]{Park}). If $\tau\colon |K|\to M/G$ is a triangulation, we say that $\tau$ is compatible with a family of subsets $\{D_i\}$ of $M$, if $\pi(D_i)$ is a union of some $\tau(\inte|\sigma|)$, where $\sigma\in K$ and $\pi\colon M\to M/G$ is the quotient map.

\begin{prop}\label{prop2}
Let $G$ be a compact Lie group, $M_0$ an algebraic $G$-variety and $\{M_j\}_{j=1}^n$ a finite system of $G$-subvarieties of $M_0$. Then there exists a semi-algebraic triangulation $\tau\colon |K|\to M/G$ compatible with the collection $\{{M_j}_{(H)}\;|\;H \text{ is a subgroup of }G\}_{j=0}^n$ where ${M_j}_{(H)}=\{x\in M_j\;|\;G_x=gHg^{-1}\text{ for some }g\in G\}$.
\end{prop}
\begin{proof}
Let $H_1,...,H_s\subset G$ be the orbit types of $G$ on $M$ and $f\colon M\to\R^d$ as above. By \cite[~Lemma 3.3]{Park} all ${M_j}_{(H_i)}$ are semi-algebraic sets, and therefore all $f({M_j}_{(H_i)})$ are also semi-algebraic. Since $i,j$ vary on finite sets, we can use Proposition \ref{Trian2} and obtain a semi-algebraic triangulation \[\lambda\colon|K|\to f(M)=\bigcup_{ij}f({M_j}_{(H_i)})\] such that each $f({M_j}_{(H_i)})$ is a finite union of some $\lambda(\inte|\sigma|)$, where $\sigma\in K$. Take $\tau=\overline{f}^{-1}\circ\lambda$.
\end{proof}

\begin{prop}\label{prop5}
Let $G$ be a compact Lie group. Let $M_0$ be an algebraic $G$-variety and $\{M_j\}_{j=1}^k$ a finite system of $G$-subvarieties. Then $M_0$ has a $G$-CW-complex structure such that each $M_j$ is a $G$-subcomplex of $M$.
\end{prop}
\begin{proof}
Let $\tau\colon |K|\to M/G$ be as in Proposition \ref{prop2} and $\pi\colon M\to M/G$ the orbit map. Let $K^\prime$ be a barycentric subdivision of $K$, which guarantees that for any simplex $\Delta^n$ of $K^\prime$, $\pi^{-1}(\tau(\Delta^n-\Delta^{n-1}))\subset {M_j}_{(H_n)}$ for some $H_n\subset G$ and $0\leq j\leq k$. Since $\tau|\colon\pi^{-1}(\tau(\Delta^n))/G\to\Delta^n$ is a homeomorphism and the orbit type of $\pi^{-1}(\tau(\Delta^n-\Delta^{n-1}))$ is constant, by \cite[~Lemma 4.4]{matumoto} there exists a continuous section $s\colon \tau(\Delta^n)\to M_j$ so that $s\circ\tau(\Delta^n-\Delta^{n-1})$ has a constant isotropy subgroup $H_n$. Consequently there is an equivariant homeomorphism \[\pi^{-1}\tau(\Delta^n-\Delta^{n-1})\cong G/H_n\times(\Delta^n-\Delta^{n-1}).\] 
Collecting $G$-cells $Gs\circ\tau(\Delta^n)$ for all simplexes of $K^\prime$ we get a $G$-CW structure over all $M_j$, $0\leq j\leq k$.
\end{proof}

For a finitely generated cosimplicial group $L$, denote $\Rep(L_n,G):=\Hom(L_n,G)/G$ and $\overline{S_t}(L_n,G):=S_t(L_n,G)/G$.

\begin{theorem}
Let $G$ be a compact Lie group and $L$ a finitely generated cosimplicial group. Then for each $n$, $\Theta(n)$ from \emph{Theorem \ref{prop16}} is a $G$-equivariant homotopy equivalence, and in particular we get homotopy equivalences
\[\Sigma \Rep(L_n,G)\simeq \bigvee_{1\leq k\leq n}\Sigma(\overline{S}_k(L_n,G)/\overline{S}_{k+1}(L_n,G)).\]
\end{theorem}
\begin{proof}
Assume $G\subset GL_N(\R)$. Under conjugation by elements of $G$, $\Hom(L_n,G)$ is an affine $G$-variety and by Lemma \ref{Lemma0} the subspaces $S_j(L_n,G)$ are $G$-subvarieties for all $1\leq j\leq n$. Hence, by Proposition \ref{prop5} $\Hom(L_n,G)$ can be given a $G$-CW-complex structure where each $S_j(L_n,G)$ is a $G$-subcomplex. Similarly, the quotient $S_k(L_n,G)/S_{k+1}(L_n,G)$ has a $G$-CW-complex structure. 

To prove that the map $\Theta(n)$ is a $G$-equivariant homotopy equivalence, first recall that conjugation by elements of $G$ defines a simplicial action on $\Hom(L,G)$, and by the naturality of each $\Theta(n)$, the $G$-equivariance follows. Let $H\subset G$ be closed subgroup. The fixed points spaces $\Hom(L_n,G)^H$ and $S_k(L_n,G)^H$ inherit a CW-complex structure so that $\Hom(L,G)^H$ is simplicially NDR. By Proposition \ref{prop15} we have homotopy equivalences
 \[\Theta(n,H)\colon\Sigma (\Hom(L_n,G)^H)\to\bigvee_{0\leq k\leq n}\Sigma (S_k(L_n,G)^H/S_{k+1}(L_n,G)^H)\]
for each $n\geq 1$. The fixed points map $\Theta(n)^H$ agrees by naturality with $\Theta(n,H)$ and thus is a homotopy equivalence. The result now follows from the equivariant Whitehead Theorem.
\end{proof}

\begin{corl}
Let $G$ be a compact Lie group. Then the homotopy equivalences in \emph{Corollary \ref{Corl1}} are $G$-equivariant homotopy equivalences, and in particular we get 
\[\Sigma \Rep(F_n/\Gamma^q_n,G)\simeq \bigvee_{1\leq k\leq n}\Sigma\left(\bigvee^{\binom{n}{k}}\Rep(F_k/\Gamma^q_k,G)/\overline{S_1}(F_k/\Gamma^q_k,G)\right).\]
\end{corl}

\begin{exam}
Let $G=SU(2)$ and $L=F/\Gamma^q$.

$\bullet$ For $q=2$, it was proven in \cite[p.~484]{Ad4} that \[\Rep(\Z^n,SU(2))/\overline{S_1}(\Z^n,SU(2))\simeq T^{\wedge n}/W=S^n/\Sigma_2\] where the action of the generating element on $\Sigma_2$ is given by \[(x_0,x_1,...,x_n)\mapsto (x_0,-x_1,...,-x_n)\]
for any $(x_0,x_1,...,x_n)$. Identifying $S^n=\Sigma S^{n-1}$, we can see the orbit space $S^n/\Sigma_2$ as first taking antipodes, and then suspending, that is $S^n/\Sigma_2\cong\Sigma\RP^{n-1}$. Thus
\[\Sigma\Rep(\Z^n,SU(2))\simeq \bigvee_{1\leq k\leq n}\Sigma\left(\bigvee^{\binom{n}{k}}\Sigma\RP^{k-1}\right).\]

$\bullet$ Let $q=3$. We have shown that $\Rep(F_n/\Gamma^3_n,SU(2))=B_n(SU(2),\{\pm I\})/G$ and using the description of these spaces given in \cite[p.~486]{Ad4}, the stable pieces are 
\[\Rep(F_n/\Gamma_n^3,SU(2))/\overline{S_1}(F_n/\Gamma^3_n,SU(2))\simeq \left(\bigvee_{K(n)} S^0\right)\vee \Sigma\RP^{n-1}.\]
where $K(n)$ is as in Example \ref{T1}. Therefore
\[\Sigma\Rep(F_n/\Gamma^3_n,SU(2))\simeq \bigvee_{1\leq k\leq n}\Sigma\left(\bigvee^{\binom{n}{k}}\left(\bigvee_{K(k)} S^0\right)\vee \Sigma\RP^{k-1}\right).\] 
\end{exam}

\section{Homotopy properties of $B(L,G)$}

\subsection{Geometric realization of $\Hom(L,G)$}

\begin{defi}
Let $L$ be finitely generated cosimplicial group and $G$ a topological group. Denote 
\[B(L,G):=|\Hom(L,G)|.\] 
\end{defi}

For the cosimplicial groups $F/\Gamma^q$ we get that $B(F/\Gamma^q,G)=B(q,G)$, the classifying space for $G$-bundles of transitional nilpotency class less than $q$. A natural question is whether or not the space $B(L,G)$ is a classifying space for a specific class of $G$-bundles. 

\begin{lemma}\label{lemma3.2}
Let $L$ be a cosimplicial group. The 1-cocycles of $L$ are in one to one correspondence with cosimplicial morphisms $F\to L$.
\end{lemma}
\begin{proof}
Suppose $b$ is a 1-cocycle. Any generator $a_j\in F_n$ is in the image of $a_1\in F_1$ under composition of coface homomorphisms, e.g., $a_j=(d^0)^{j-1}(d^2)^{n-j}(a_1)$ for all $j\geq 1$. Define $h^n\colon F_n\to L_n$ as $h^n(a_j)=(d^0)^{j-1}(d^2)^{n-j}(b)$. To show that $h$ is cosimplicial, consider the diagrams
\[\xymatrix{F_{n-1}\ar[r]^{h^{n-1}}\ar[d]_{d^i}&L_{n-1}\ar[d]^{d^i}\\F_n\ar[r]^{h^n}&L_n}\;\;\;\;\;\xymatrix{F_{n+1}\ar[r]^{h^{n+1}}\ar[d]_{s^i}&L_{n+1}\ar[d]^{s^i}\\F_{n}\ar[r]^{h^n}&L_n}.\]
We prove the case of coface homomorphisms. Let $a_j\in F_{n-1}$. On one side we get
\[h_nd^i(a_j)=\left\{\begin{matrix}
(d^0)^{j-1}(d^2)^{n-j}(b)&j<i\\
(d^0)^{j-1}(d^2)^{n-j}(b)(d^0)^{j}(d^2)^{n-j-1}(b)&j=i\\
(d^0)^{j}(d^2)^{n-j-1}(b)&j>i
\end{matrix}\right.\]
and applying the cosimplicial identity $d^kd^l=d^ld^{k-1}$ where $k>l$ we obtain
\[d^ih^{n-1}(a_j)=\left\{\begin{matrix}
(d^0)^{j-1}(d^{i-j+1})(d^2)^{n-j-1}(b)&j<i\\
(d^0)^{j-1}(d^1)(d^2)^{n-j-1}(b)&j=i\\
(d^0)^{j}(d^2)^{n-j-1}(b)&j>i
\end{matrix}\right..\]
We need to analyze 2 cases: 

$\bullet$ $j<i$ implies that $i-j+1\geq2$, thus $d^{i-j+1}(d^2)^{n-j-1}=(d^2)^{n-j}$.

$\bullet$ $j=i$. The equality follows from equation \ref{eq2} applied to $(d^2)^{n-j-1}(b)=b_{n-j}$.

Commutativity for the codegeneracy homomorphisms is similar, but using the cosimplicial identity $s^kd^l=d^ls^{k-1}$ with $k>l$ and condition \ref{eq3} above. Hence $h$ is uniquely determined by $b$. Given a morphism $F\to L$, the element $b$ is given by the image of $a_1\in F_1$. 
\end{proof}

\begin{prop}\label{prop3.4}
Let $L$ be a cosimplicial group and $h_b\colon F\to L$ be the morphism defined on $F_1\to L_1$ as $a_1\mapsto b$. Then the diagram
\[\xymatrix{F\ar@{^(->}[r]^{\iota_+}\ar[d]_{h_b}&F^+\ar[d]^{Id*h_b}\\L\ar@{^(->}[r]^{\iota_b}&L^b}\]
is a pushout of cosimplicial groups.
\end{prop}
\begin{proof}
Suppose $f\colon F^+\to K$ and $g\colon L\to K$ are morphisms such that $f\circ \iota=g\circ h_b$. Define $h\colon L^b\to K$ on each $L_n^b=\overline{F}_0*L_n$ as $h^n(a_0)=f^n(a_0)$ (here $f^n$ is evaluated on $a_0\in \overline{F}_0*F_n$) and $h^n(x)=g^n(x)$ for any $x\in L_n$. To check that $h$ is in fact a cosimplicial homomorphism, by construction of $L^b$ and $h$, we just need to verify commutativity with coface maps at level $i=0$. Consider 
\[\xymatrix{L^b_{n-1}\ar[r]^{h^{n-1}}\ar[d]_{d^0_b}&K_{n-1}\ar[d]^{d^0}\\L^b_n\ar[r]^{h^n}&K_n}.\]
We only need to see what happens at $a_0\in L_{n-1}^b:$ 
\begin{align*}
d^0h^{n-1}(a_0)=d^0f^{n-1}(a_0)=&f^nd^0_+(a_0)=f^n(a_0a_1)=f^n(a_0)f^n(a_1)\text{ and}\\ 
h^{n}d^0_+(a_0)=&h^n(a_0b_n)=f^n(a_0)g^n(b_n).
\end{align*}
Denote $b=b_1$. By hypothesis $g^1(b_1)=f^1(a_1)$. Since
\begin{align*}
 g^n(b_n)=&(d^2)^{n-1}g^1(b_1)\text{ and}\\
 f^n(a_1)=&(d^2)^{n-1}f^1(a_1),
\end{align*}
the desired equality holds.
\end{proof}

\begin{corl}
Let $G$ be a well based topological group and $L$ a finitely generated cosimplicial group. Using the notation above, suppose $h_b\colon F\to L$ is a morphism. Then the inclusion $\iota_b\colon L\hookrightarrow L^b$ defines a principal $G$-bundle $|\iota_b^*|\colon B(L^b,G)\to B(L,G).$
\end{corl}
\begin{proof}
From the pushout diagram in Proposition \ref{prop3.4}, and applying the functors $\Hom(\_,G)$ and geometric realization, we obtain the pullback diagram
\[\xymatrix{B(L^b,G)\ar[r]\ar[d]_{|\iota_b^*|}&EG\ar[d]\\B(L,G)\ar[r]^{|h_b^*|}&BG}\]
and hence $|\iota_b^*|$ is a principal $G$-bundle.
\end{proof}

\begin{exam}We have seen that there is only one non-constant homomorphism $h_{a_1}=Id\colon F\to F$. For $q>2$ it can be shown that the same is true for $L= F/\Gamma^q$, where $h_{a_1}\colon F\to F/\Gamma^q$ at each $n$ is the quotient homomorphism. The corresponding $B((F/\Gamma^q)^+,G)$ is the space $E(q,G)$ defined in \cite[p.~94]{Ad5}, and $|h_{a_1}^*|\colon B(q,G)\to BG$ is the inclusion. The bundle $E(q,G)\to B(q,G)$ classifies transitionally nilpotent bundles of class less than $q$ (see \cite[section~5]{Ad3}). The case $q=2$ is more interesting since $Z^1(F/\Gamma^2)=\Z$. For $m=1$ we obtain $B((F/\Gamma^2)^+,G)=E(2,G)$ and $E(2,G)\to B(2,G)$ classifies transitionally commutative bundles (see \cite[section~2]{Ad1}). Since multiplication by $-1$ induces a cosimplicial automorphism of $F/\Gamma^2$, all constructions are equivalent for $m=-1$. Now let $m>1$. The bundle $B((F/\Gamma^2)^m,G)\to B(2,G)$ will classify $G$-bundles whose transition functions $g_{\alpha\beta}\colon U_\alpha\cap U_\beta\to G$ factor through
\[\xymatrix{U_\alpha\cap U_\beta\ar[r]^{\rho_{\alpha,\beta}}\ar[rd]_{g_{\alpha\beta}}&G\ar[d]^m\\&G}\]where $\rho_{\alpha\beta}$ are transitionally commutative and $m$ denotes taking the $m$-th power of elements in $G$.\end{exam}

\subsection{Relation between commutative $\mathbb{I}$-monoids and infinite loop spaces}

In this section we recall briefly the notion of $\mathbb{I}$-monoid and how it is related to infinite loop spaces. This is more widely covered in \cite{Ad3}. Our goal is to use this machinery to show that for a finitely generated cosimplicial group $L$,  $B(L,U)=\colim{n}B(L,U(n))$ is a non-unital $E_\infty$-ring space when $\Hom(L_0,U)$ is path connected.

Let $\I$ stand for the category whose objects are the sets $[0]=\emptyset$ and $[n]=\{1,...,n\}$ for each $n\geq1$, and morphisms are injective functions. Any morphism $j\colon [n]\to[m]$ in $\I$ can be factored as a canonical inclusion $[n]\hookrightarrow [m]$ and a permutation $\sigma\in \Sigma_m$. This category is symmetric monoidal under two different operations, namely, concatenation $[n]\sqcup[m]=[n+m]$ with symmetry morphism the permutation $\tau_{m,n}\in\Sigma_{n+m}$ defined as
\[\tau_{m,n}(i)=\left\{\begin{matrix}
n+i&\text{ if }i\leq m\\
i-m&\text{ if }i>m
\end{matrix}\right.\] 
and identity object $[0]$. The second operation is Cartesian product $[m]\times[n]=[mn]$ with $\tau_{m,n}^{\times}\in\Sigma_{mn}$ given by
\[\tau_{m,n}^{\times}((i-1)n+j)=(j-1)m+i\]  
where $1\leq i\leq m$ and $1\leq j\leq n$. In this case the identity object is $[1]$. Cartesian product is distributive under concatenation (both left and right).

\begin{defi}
An $\I$-\emph{space} is a functor $X\colon\I\to\Top$. This functor is determined by the following. 
\begin{enumerate}
\item A family of spaces $\{X[n]\}_{n \geq 0}$, where each $X[n]$ is a $\Sigma_n$-space; 
\item $\Sigma_n$-equivariant structural maps $j_n\colon X[n]\to X[n+1]$ (here we consider $X[n+1]$ is a $\Sigma_{n}$-space under the restriction of the $\Sigma_{n+1}$-action to the canonical inclusion $\Sigma_n\hookrightarrow \Sigma_{n+1}$) with the property: for any $j\colon[n]\to[m]$, $\sigma,\sigma^\prime\in\Sigma_m$ whose restrictions in $\Sigma_n$ are equal, we have $\sigma\cdot x=\sigma^\prime\cdot x\in X(j)(X[n])$. 
\end{enumerate}
\end{defi}

We say that an $\I$-space $X$ is a \emph{commutative $\I$-monoid} if it is a symmetric monoidal functor $X\colon(\I,\sqcup,[0])\to(\Top,\times,\{pt\})$. Additionally, we say that $X$ is a \emph{commutative $\I$-rig} if $X$ is also symmetric monoidal with respect to $(\I,\times,[1])$. For the latter definition we also require $X$ to preserve distributivity.

\begin{defi}
Let $\bf{C}$ be a small category and $Y\colon {\bf{C}}\to \Top$ a functor. Denote by ${\bf{C}}\ltimes Y$ the category of elements of $Y$, that is, objects are pairs $(c,x)$ consisting of an object $c$ of $\bf{C}$ and a point $x\in Y(c)$. A morphism in ${\bf{C}}\ltimes Y$ from $(c,x)$ to $(c^\prime,x^\prime)$ is a morphism $\alpha\colon c \to c^\prime$ in $\bf{C}$ satisfying the equation $Y (\alpha)(x) = x^\prime$.
 \end{defi}

Given $Y\colon{\bf C}\to\Top$, with the notation above, if we consider ${\bf{C}}\ltimes Y$ as a topological category whose space of objects and space of morphisms are \[\bigsqcup_{c\in{\text{obj}(\bf{C})}}Y(c)\text{ and }\bigsqcup_{f\in{\text{mor}(\bf{C})}}Y(f),\]
then we have that the homotopy colimit of $Y$ is the classifying space $B({\bf C}\ltimes Y)=\hocolim_{{\bf C}}Y,$ that is, the realization of the nerve of the category ${\bf C}\ltimes Y$. 

Let $X$ denote a commutative $\I$-monoid. The category of elements $\I\ltimes X$ is a  permutative category, that is, a symmetric monoidal category where associativity holds on the nose. According to \cite{May 2}, the classifying space of a  permutative category has an $E_\infty$-space structure, and so we get that $\hocolim_{{\I}}X$ has an $E_\infty$-space structure. Here we think of an $E_\infty$-space as a space with an operation that is associative and commutative up to a system of coherent homotopies. Thus, the group completion $\Omega B(\hocolim_{{\I}}X)$ is an infinite loop space. If $X$ is a commutative $\I$-rig, then $\I\ltimes X$ is a bipermutative category and its classifying space is an $E_\infty$-ring space (as explained in \cite{Ad3}), that is, an $E_\infty$-space with an operation that is associative and commutative (up to coherent homotopy) that is distributive (up to coherent homotopy) over the $E_\infty$-space operation.

Consider the subcategory of $\I$ consisting of the same set of objects and all isomorphisms. We denote it as $\mathbb{P}$. The (bi) permutative structure on $\I\ltimes X$ restricts to $\mathbb{P}\ltimes X$, so that $\hocolim_{\mathbb{P}}X$ is also an $E_\infty$-space ($E_\infty$-ring space) and its group completion $\Omega B(\hocolim_{\mathbb{P}}X)$ is an infinite loop space ($E_\infty$-ring space). The maps $X[n]\to *$ induce a map of (bi) permutative categories $\mathbb{P}\ltimes X\to\mathbb{P}\ltimes *$ and therefore a map of infinite loop spaces ($E_\infty$-ring spaces)
\[\rho^X\colon \Omega B(\hocolim_{\mathbb{P}}X)\to \Omega B(\hocolim_{\mathbb{P}}*).\]
It follows that the homotopy fiber $\text{hofib}\;\rho^X$ is an infinite loop space (non-unital $E_\infty$-ring space). Denote $X_\infty:=\hocolim_\mathbb{N}X$ where $\mathbb{N}$ denotes the subcategory of $\I$ with same set of objects and as arrows the canonical inclusions, and $X^+_\infty$ its Quillen plus construction applied with respect to the maximal perfect subgroup of $\pi_1(X_\infty)$. The following proposition is proved in \cite[Theorem~3.1]{Ad3}.

\begin{prop}\label{prop2.2.1}
Let $X\colon \I\to \Top$ be a commutative $\I$-monoid. Assume that

$\bullet$ the action of $\Sigma_\infty$ on $H_*(X_\infty)$ is trivial;

$\bullet$ the inclusions induce natural isomorphisms $\pi_0(X[n])\simeq \pi_0(X_\infty)$ of finitely generated abelian groups with multiplication compatible with the Pontrjagin product and in the center of the homology Pontrjagin ring;

$\bullet$ the commutator subgroup of $\pi_1(X_\infty)$ is perfect (for each component) and $X_\infty^+$ is abelian.

\noindent Then $\text{hofib}\;\rho^X\simeq X_\infty^+$, and in particular $X_\infty^+$ is an infinite loop space. 
\end{prop}

Note that the last two conditions of the previous Proposition are satisfied when each $X[n]$ is connected and $X_\infty$ is abelian. Under these hypothesis $X_\infty$ has an infinite loop space structure.

\subsection{non unital $E_\infty$-ring space structure of $B(L,U)$}

Our first example and application of the machinery described in the previous section is showing the classical result \[\colim{m}U(m)=U\] has an infinite loop space structure. This will allow us to prove non-unital $E_\infty$-ring space structures on our spaces of interest. 
 
First, we show that $U(\_)$ is a commutative $\I$-rig. Recall that $\Sigma_m\subset U(m)$ as permutation matrices, so that $U(m)$ has a $\Sigma_m$ action. Consider the inclusions $i_m\colon U(m)\to U(m+1)$
\[A\mapsto \left(\begin{matrix}
A&0\\
0&1
\end{matrix}\right)\] 
which are continuous and preserve group structure. The maps $i_m$ restrict to the canonical inclusions $\Sigma_m\hookrightarrow \Sigma_{m+1}$, therefore 
\[i_m(\sigma\cdot A)=i_m(\sigma)i_m(A)i_m(\sigma)^{-1}=\sigma\cdot i_m(A),\]
where $\sigma\in \Sigma_m$, $A\in U(m)$ and on the right hand side $\sigma\in\Sigma_{m+1}$. Now let $\sigma,\sigma^\prime\in\Sigma_r$, $m<r$ and suppose both restrictions to the subset $\{1,...,m\}$ determine equal permutations in $\Sigma_m$. Denote $i=i_r\circ i_{r-1}\circ\cdots\circ i_m$. Then, for $A\in U(m)$, 
\[\sigma\cdot i(A)=\left(\begin{matrix}
(\sigma|_m)A(\sigma|_{m}^{-1})&0\\
0&I_{r-m}
\end{matrix}\right)=\left(\begin{matrix}
(\sigma^\prime|_m)A(\sigma^\prime|_{m}^{-1})&0\\
0&I_{r-m}
\end{matrix}\right)=\sigma^\prime\cdot i(A).\]
Therefore $U(\_)\colon \mathbb{I}\to \Top$ is a functor. This $\mathbb{I}$-space has a commutative $\mathbb{I}$-rig structure as follows. Let $\oplus_{m,n}\colon U(m)\times U(n)\to U(m+n)$ denote the block sum of matrices, which is a group homomorphism. 
The $(m,n)$ shuffle map $U(n+m)\to U(n+m)$ is given by $A\mapsto \tau_{m,n}\cdot A$. We have the commutative diagram
\[\xymatrix{U(m)\times U(n)\ar[d]^{\tau}\ar[r]^{\oplus_{m,n}}& U(m+n)\ar[d]^{\tau_{m,n}}\\ U(n)\times U(m)\ar[r]^{\oplus_{n,m}}&U(m+n)}\]
where $\tau(A,B)=(B,A)$. Therefore $U(\_)$ is a commutative $\I$-monoid. The other monoidal structure is given by $\otimes_{m,n}\colon U(m)\times U(n)\to U(mn)$ the tensor product of matrices. Indeed, by definition $\tau_{m,n}^{\times}\cdot \otimes_{m,n}(A,B)=\otimes_{n,m}\tau(A,B)$, where $A\in U(m)$ and $B\in U(n)$. Since the image $\oplus_{m,n}(U(m)\times U(n))$ correspond to direct sum, then associativity, left and right distributivity over $\otimes_{m,n}$ hold. 

Now we check the conditions of Proposition \ref{prop2.2.1}: the action of $\Sigma_m$ on $U(m)$ is homologically trivial since conjugation action on $U(m)$ is trivial up to homotopy, being $U(m)$ path connected. The inclusions $i_m$ are cellular and hence $U(\_)_\infty\simeq U$ and since $U$ is an $H$-space under block sum of matrices, it is abelian. Therefore $U(\_)_\infty\simeq U$ is an infinite loop space (non-unital $E_\infty$-ring space).

\begin{lemma}\label{lemma1s2.1}
Let $L$ be a finitely generated cosimplicial group and $G,H$ real algebraic linear groups. Let $p_1\colon G\times H\to G$ and $p_2\colon G\times H\to H$ be the projections. Then
\[B(L,p_1)\times B(L,p_2)\colon B(L,G\times H)\to B(L,G)\times B(L,H)\]
is a natural homeomorphism.
\end{lemma}
\begin{proof}
Since $G\times H$ is a direct product, $p_1$ and $p_2$ are continuous homomorphism and therefore \[p=(p_1)_*\times(p_2)_*\colon \Hom(L,G\times H)\to \Hom(L,G)\times \Hom(L,H)\]
is a simplicial map. Its easy to check that in fact is a simplicial isomorphism. Both $G$ and $H$ being real algebraic, imply that $\Hom(L_n,G)$ and $\Hom(L_n,H)$ have a CW-complex structure, and therefore are $k$-spaces. By \cite[~Theorem 11.5]{May 1} the composition
\[\xymatrix{B(L,G\times H)\ar[r]^{|p|\;\;\;\;\;\;\;\;\;\;\;\;\;\;\;\;}& \Hom(L,G)\times \Hom(L,H)|\;\ar[r]^{\;\;\;\;\;\;\;\;|\pi_1|\times|\pi_2|}& B(L,G)\times B(L,H)}\]
is a natural homeomorphism where $|\pi_1\circ p|\times|\pi_2\circ p|=B(L,p_1)\times B(L,p_2)$.
\end{proof}

\begin{prop}\label{prop2.3.1}
Let $L$ be a finitely generated cosimplicial group, then $B(L,U(\_))$ is a commutative $\mathbb{I}$-rig.
\end{prop}
\begin{proof}
Consider the $\mathbb{I}$-rig $U(\_)$. Both the structural maps $i_m$ and the action by elements of $\Sigma_m$ are continuous group homomorphisms and hence $B(L,U(\_))=B(L,\_) U(\_)$ is an $\mathbb{I}$-space. Also, block sum of matrices and tensor product are topological group morphisms so that with Lemma \ref{lemma1s2.1} we can define \[\mu_{m,n}=B(L,\oplus_{m,n})\circ(B(L,p_1)\times B(L,p_2))^{-1} \text{ and } \pi_{m,n}=B(L,\otimes_{m,n})\circ(B(L,p_1)\times B(L,p_2))^{-1},\]
where $p_1\colon U(m)\times U(n)\to U(m)$ and $p_2\colon U(m)\times U(n)\to U(n)$ are the projections. Let $p_1^\prime\colon U(n)\times U(m)\to U(n)$ and $p_2^\prime\colon U(n)\times U(m)\to U(m)$ denote also projections. Notice that \[\tau \circ B(L,p_2^\prime)\times B(L,p_1^\prime)=B(L,p_1)\times B(L,p_2)\circ B(L,\tau)\] (where $\tau$ as before, is the symmetry morphism in $\Top$). This implies that all properties satisfied by $\oplus_{m,n}$ and $\otimes_{m,n}$ will be preserved by $\mu_{m,n}$ and $\pi_{m,n}$. 
\end{proof}

\begin{theorem}
Let $L$ be finitely generated cosimplicial group and suppose that the space $\Hom(L_0,U(m))$ is path connected for all $m\geq 1$. Then, $B(L,U)$ is a non-unital $E_\infty$-ring space.
\end{theorem}
\begin{proof}
By Proposition \ref{prop2.3.1}, $B(L,U(\_))$ is a commutative $\I$-rig. It remains to check the conditions of Proposition \ref{prop2.2.1}. Note that the conjugation action of $\Sigma_n$ is homologically trivial since it factors through conjugation action on $U(m)$. Since all $\Hom(L_0,U(m))$ are path connected, $|\Hom(L,U(m))|=B(L,U(m))$ is path connected for all $m\geq1$. The colimit $B(L,U)$ is also an $H$-space under block sum of matrices, and therefore abelian. 
\end{proof}

\begin{exam}
The property $\pi_0(\Hom(L_0,U(m)))=0$ for all $m\geq1$ is satisfied by the following cosimplicial groups:

$\bullet$ $L=F/\Gamma^q$ and $L=F/F^{(q)}$ since $L_0=\{e\}$ in both cases. 

$\bullet$ $L=\overline{F}/\Gamma^q$ and $L=\overline{F}/F^{(q)}$ since $\Hom(L_0,U(m))=U(m)$ in both cases.

$\bullet$ Consider $\Sigma_{2,3}$, and the cosimplicial morphism $h_{\sigma_1\sigma_2}\colon F\to \Sigma_{2,3}$. The image $h_{\sigma_1\sigma_2}(F)$ defines a cosimplicial subgroup of $\Sigma_{2,3}$, such that $h_{\sigma_1\sigma_2}(F)_0=\{e\}$.
\end{exam}

\begin{remark}
The results in this section also apply for the groups $SU$ and $Sp$. For $SO$ and $O$ the proofs are not exactly similar, but still true. The arguments used in \cite[~Theorem 4.1]{Ad3} also apply in our case. 
\end{remark}

{\small\sc\footnotesize Department of Mathematics, University of British Columbia, Vancouver BC V6T 1Z2, Canada}

{\it\footnotesize E-mail adress:} \footnotesize{\texttt bernvh@math.ubc.ca}

\end{document}